\documentclass{article}
\usepackage{color,bbold, hyperref}
\usepackage{amssymb,amsmath,amsthm,graphicx, color, amsfonts,tabularx}
\usepackage[english]{babel}
\usepackage{afterpage}
\usepackage{authblk, relsize}
\usepackage{float}
\usepackage{geometry}
\def\supp{\mathop{\rm supp}\nolimits}

\newtheorem{theorem}{Theorem}[section]

\newtheorem{lemma}[theorem]{Lemma}
\newtheorem{proposition}[theorem]{Proposition}

\newtheorem{definition}[theorem]{Definition}
\newtheorem{remark}[theorem]{Remark}

\newcommand{\R}{\mathbb{R}}

\renewenvironment{proof}[1][.]{%
\bigskip\noindent{\bf Proof#1 }}{%
\hfill$\blacksquare$\bigskip}

\newcommand{\mS}{\mathcal S}
\newcommand{\mM}{\mathcal M}

\newcommand{\on}{\operatorname}

\begin{document}
\pagestyle{myheadings}
\title{The Hutchinson-Barnsley theory for generalized iterated function systems by means of infinite iterated function systems}
\author[1]{Elismar R. Oliveira
\thanks{E-mail: elismar.oliveira@ufrgs.br}}
\affil[1]{Universidade Federal do Rio Grande do Sul}

\date{\today}
\maketitle

\begin{abstract}
The study of generalized iterated function systems (GIFS) was introduced by Mihail and Miculescu in 2008. We provide a new approach to study those systems as the limit of the Hutchinson-Barnsley  {setting for infinite iterated function systems (IIFS) which has been developed} by many authors in the last years. We show  that any attractor of a contractive generalized iterated function system is the limit {with respect to} Hausdorff-Pompeiu metric of attractors of contractive infinite iterated function systems. We also prove that any Hutchinson measure for a contractive generalized iterated function system with probabilities is the limit with respect to the Monge-Kantorovich metric of the Hutchinson measures for contractive infinite iterated function systems with probabilities.
\end{abstract}

\section{Introduction}
The study of generalized iterated function systems, introduced in 2008 by Mihail and Miculescu~\cite{Mihail2008}, provided an interesting perspective of how to {produce fractal sets that extends the classical one related to contractive iterated function systems}. In 2015 Strobin~\cite{strobin2015attractors} has proved  the existence of attractors of generalized iterated function systems of order $m$ which are not attractor of any generalized iterated  function system of order $m-1$, in particular of any contractive  finite iterated function system (which is also a generalized iterated function systems of order $1$).

Many work has been done to directly attack the problem of  developing the Hutchinson-Barnsley theory for generalized iterated function systems by direct means such as fixed point theorems and code spaces. It is worth to notice two very recent works published in the last two years. The first one  {is due to Strobin~\cite{Strobin2020}}, proving that a generalized iterated function system with probabilities consisting of generalized contractive maps generates the unique generalized Hutchinson measure  {extending the recent results due to Miculescu and Mihail}. The analogous problem for the place dependent case  {has been} leaved open (see {\cite[Problem 5.3]{Strobin2020}}). The second one  {has been} Guzik and Kapica~\cite{Guzik_2021} where a criteria on the existence of a unique attracting probability measure for stochastic process induced by generalized iterated function systems with (fixed) probabilities is developed. Both works {require} intricate and profound machinery to achieve these results.

Seeking to obtain some dynamical meaning, and hopefully an elementary way to address these problems, we employ another approach, proving that any attractor of a contractive generalized iterated function system is the limit,  {with respect to} Hausdorff-Pompeiu metric, of attractors of contractive infinite iterated function systems and that any Hutchinson measure for a contractive generalized iterated function system with probabilities is the limit  {with respect to} the Monge-Kantorovich metric of the Hutchinson measures for contractive infinite iterated function systems with probabilities.  Regarding to fractal generation, we show that it is enough to study (possibly infinite, uncountable) iterated function systems instead generalized iterated function systems. As a  byproduct we briefly present some new iterative procedures to approximate the attractor and the Hutchinson measure of a generalized iterated function system via evaluation maps.  {In the paper we deal with Banach contractive type function systems, but most cited results can be extended to function systems consisting of maps satisfying weaker contractive conditions.}

The paper is organized as follows:\\
In Section~\ref{sec:Infinite IFS} we recall some basic facts on infinite iterated function systems, in particular the sufficient conditions  {for} the existence of an attractor from \cite{Dumitru2013AttractorsOI}. In Section~\ref{sec:reduction} we show how to introduce an infinite iterated function system associated to a given generalized iterated function system and a closed and bounded set. Then we prove that it  {satisfies sufficient conditions for generating an attractor and we define} an evaluation map which assign this attractor to the  chosen set. Finally, we show that the evaluation map is a contraction on the space of nonempty bounded and closed sets and that the attractor of the given generalized iterated function system is a fixed point of  {this contraction map}. This proves our first main result, Theorem~\ref{thm:main-thm}.  {Then}, we introduce the joint evaluation map for sets and probabilities proving that it is also a contraction and has the attractor and the Hutchinson measure of the generalized iterated function system as its fixed point, which is the content of our main result, Theorem~\ref{thm:main_theorem_measure}. For the last, in the Section~\ref{sec:concluding-remarks} we explain some additional facts and discuss some possibilities for future work.

As this work is a connection between different areas of research on iterated function systems,  {for the reader convenience} we provide a detailed set of known results on the Hutchinson-Barnsley theory along of the text and a bibliographical review.

\section{Preliminaries and Infinite IFS}\label{sec:Infinite IFS}
First we recall some basic facts on sets and metric spaces {,} necessary to study infinite iterated function systems. Our main reference is \cite{Dumitru2013AttractorsOI}, however many authors have studied this subject in the last few years {, see \cite{lukawska_jachymski_2005}, \cite{hyong2005ergodic}, \cite{Systems2008}, \cite{Miculescu2012}, \cite{secelean2002sufficient}, \cite{Secelean2014}, \cite{Secelean2011}, \cite{Sec01}, \cite{Secelean2014Invar} and \cite{CIM2014}. From now on,} $(X, d)$ is always a metric space. We will denote by $\mathcal{U}^{*}(X)$ the set of nonempty subsets of $X$,  $\mathcal{K}^{*}(X)$ the set of nonempty compact subsets of $X$, $\mathcal{B}^{*}(X)$ the set of nonempty bounded closed subsets of $X$.  {As usual, we denote by $\bar{A}$ the closure of $A$, with respect to the topology of $(X, d)$.}

The generalized Hausdorff-Pompeiu semi-distance is the  {function} $$h: \mathcal{U}^{*}(X) \times \mathcal{U}^{*}(X) \to [0,+\infty]$$ defined by $h(A, B)=\max \{d(A, B), d(B, A)\}$, where $$\displaystyle d(A, B)=\sup _{x \in A} d(x, B)=\sup _{x \in A}\left(\inf _{y \in B} d(x, y)\right).$$

 {We will be mostly interested in spaces} $\left(\mathcal{K}^{*}(X), h\right)$ and $\left(\mathcal{B}^{*}(X), h\right)$. For a function $\psi: X \rightarrow X$ we denote by $\operatorname{Lip}(\psi) \in[0,+\infty]$ the Lipschitz constant associated to $\psi$, which is given by
$\displaystyle
\operatorname{Lip}(\psi)=\sup _{x, y \in X ; x \neq y} \frac{d(\psi(x), \psi(y))}{d(x, y)}.
$
We say that $\psi$ is a Lipschitz function if $\operatorname{Lip}(\psi)<+\infty$ and a contraction if $\operatorname{Lip}(\psi)<1$.

\begin{theorem}[\cite{Barnsley1993}]\label{thm:metric-spaces-of-sets}
In the above frame, the following conditions hold
\begin{enumerate}
  \item  $\left(\mathcal{B}^{*}(X), h\right)$ and $\left(\mathcal{K}^{*}(X), h\right)$ are metric spaces  {and $\left(\mathcal{K}^{*}(X), h\right)$ is closed in} $\left(\mathcal{B}^{*}(X), h\right)$.
  \item   If $(X, d)$ is complete, then $\left(\mathcal{B}^{*}(X), h\right)$ and $\left(\mathcal{K}^{*}(X), h\right)$ are complete metric spaces.
  \item   If $(X, d)$ is compact, then $\left(\mathcal{K}^{*}(X), h\right)$ is compact and in this case $\mathcal{B}^{*}(X)=\mathcal{K}^{*}(X)$.  {
  \item   If $(X, d)$ is separable, then  $\left(\mathcal{K}^{*}(X), h\right)$ is separable.
  \item If $H$ and $K$ are two nonempty subsets of $X$ then $h(H, K)=h(\bar{H}, \bar{K})$.
  \item If $\left(H_{\theta}\right)_{\theta \in \Theta}$ and $\left(K_{\theta}\right)_{\theta \in \Theta}$ are two families of nonempty subsets of $X$ then
$$
h\left(\bigcup_{\theta \in \Theta} H_{\theta}, \bigcup_{\theta \in \Theta} K_{\theta}\right)=h\left(\bigcup_{\theta \in \Theta} H_{\theta}, \overline{\bigcup_{\theta \in \Theta} K_{\theta}}\right) \leq \sup _{\theta \in \Theta} h\left(H_{\theta}, K_{\theta}\right).
$$
  \item If $H$ and $K$ are two nonempty subsets of $X$ and $\psi: X \rightarrow X$ is a Lipschitz function then $h(\psi(K), \psi(H)) \leq \operatorname{Lip}(\psi) \cdot h(K, H)$.}
\end{enumerate}
\end{theorem}

\begin{definition}\label{def:bounded-family} A family of continuous functions $\left(\psi_{\theta}\right)_{\theta \in \Theta}, \psi_{\theta}: X \to X$  is said to be \emph{bounded} if for every bounded set $A \subset X$ {,} the set $\bigcup_{\theta \in \Theta} \psi_{\theta}(A)$ is bounded.
\end{definition}

\begin{definition}\label{def:IIFS} An \emph{infinite iterated function system} (IIFS in short) on $X$ consists of a bounded family of continuous functions $\left(\psi_{\theta}\right)_{\theta \in \Theta}$ on $X$,  {and it is denoted} by $\mathcal{R}=\left(X,\left(\psi_{\theta}\right)_{\theta \in \Theta}\right)$. When $\Theta$ is finite we obtain the classical notion  of an iterated function system (IFS).
\end{definition}

\begin{definition}\label{def:fractal-op-IIFS} For an IIFS $\mathcal{R}=\left(X,\left(\psi_{\theta}\right)_{\theta \in \Theta}\right)$, the {\emph{fractal operator}(or Hutchinson-Barnsley operator)} $F_{\mathcal{R}}: \mathcal{B}^{*}(X) \rightarrow \mathcal{B}^{*}(X)$ is the function defined by $F_{\mathcal{R}}(A)=\overline{\bigcup_{\theta \in \Theta} \psi_{\theta}(A)}$ for every $A \in \mathcal{B}^{*}(X)$.
\end{definition}
The closure in the above definition is necessary because an arbitrary union of closed sets may not be closed.
\begin{remark} {
    {It is a classical result in the Hutchinson-Barnsley theory that if the functions} $\psi_{\theta}$ are contractions, for every $\theta \in \Theta$ with $\operatorname{sup}_{\theta \in \Theta} \operatorname{Lip}\left(\psi_{\theta}\right)<1$, then the function $F_{\mathcal{R}}$ is a contraction and verifies $\operatorname{Lip}\left(F_{\mathcal{R}}\right) \leq \sup _{\theta \in \Theta} \operatorname{Lip}\left(\psi_{\theta}\right)<1$.}
\end{remark}
 {From the contractivity of $F_{\mathcal{R}}$ we obtain, via fixed point theorem for contractions, the following existence result (see \cite[Theorem 1.2]{Dumitru2013AttractorsOI} or \cite[Theorem 3.2 and Theorem 4.1]{Lewellen-93} for details):}
\begin{theorem} \label{thm:exist-attract-IIFS}
Let $(X, d)$ be a complete metric space and $\mathcal{R}=\left(X,\left(\psi_{\theta}\right)_{\theta \in \Theta}\right)$  {an  IIFS, such that $\displaystyle\alpha=\sup _{\theta \in \Theta} \operatorname{Lip}\left(\psi_{\theta}\right)<1$.} Then there exists a unique set $A_{\mathcal{R}} \in \mathcal{B}^{*}(X)$ such that $F_{\mathcal{R}}(A_{\mathcal{R}})=A_{\mathcal{R}}$  {and } for any $A_{0} \in \mathcal{B}^{*}(X)$ the sequence $\left(A_{k}\right)_{k \in \mathbb{N}}$ defined by $A_{k+1}=F_{\mathcal{R}}\left(A_{k}\right)$ is convergent to $A_{\mathcal{R}}$ {with respect to} the metric $h$.  {Moreover, for every $k \in \mathbb{N}$, we have}, $ h\left(A_{k}, A_{\mathcal{R}}\right) \leq \frac{\alpha^{k}}{1-\alpha}\; h\left(H_{0}, H_{1}\right)$.
\end{theorem}

\begin{definition}
   The unique set $A_{\mathcal{R}} \in \mathcal{B}^{*}(X)$, given by Theorem~\ref{thm:exist-attract-IIFS} is called the  {attractor of the IIFS $\mathcal{R}$.}
\end{definition}

\section{Hutchinson-Barnsley theory for GIFS using the induced  IIFS}\label{sec:reduction}

\subsection{Attractors of GIFS}
 {We consider the maximum distance in $X^{m}$, that is, given $x,y \in X^{m}$ we have $d_{m}(x,y)=\max_{1\leq i \leq m} d(x_{i},y_{i}).$ Given a function $\phi:X^{m} \to X$, we define
$$
\operatorname{Lip}(\phi)=\inf\{k \; : \;  d(F(x), F(y)) \leq k \; d_{m}(x, y) \text{ for all } x, y \in X^{m}\}.
$$
When $\operatorname{Lip}(\phi)< \infty$  say that that $\phi$ is a \emph{Lipschitz function} and $\operatorname{Lip}(\phi)$ is the \emph{Lipschitz constant} of $\phi$. If $\operatorname{Lip}(\phi)<1$, then $\phi$ is called a \emph{Lipschitz contraction}.}

 {Generalized iterated function systems were introduced  {in} \cite{Mihail2008} as follows:
\begin{definition}\label{def:GIFS}
    Let $m \in \mathbb{N}^{*}$. A \emph{generalized iterated function system} (GIFS, for short) on $X$, of order $m$, denoted by $\mathcal{S}=\left(X^{m}, \left(\phi_{j}\right)_{j\in\{1, ..., n\}}\right)$, consists of a finite family of Lipschitz maps $\phi_{j}: X^{m} \rightarrow X$. The \emph{generalized Hutchinson operator} $F_{\mathcal{S}}:(\mathcal{B}^{*}(X))^m \to \mathcal{B}^{*}(X)$ is given by
$$
F_{\mathcal{S}}(A_{1} , \ldots, A_{m})= \bigcup_{j\in\{1,...,n\}} \phi_{j}(A_{1} \times \ldots \times A_{m}),
$$
for every $A_{1} , \ldots, A_{m} \in \mathcal{B}^{*}(X)$.
\end{definition}}

\begin{theorem}[\cite{Mihail2008} or \cite{Strobin2013}] \label{thm:GIFS-attractor-existence}
Let $X$ be a complete metric space and $$\mathcal{S}=\left(X^m, (\phi_j)_{j\in\{1,...,n\}}\right)$$ be a GIFS of order $m$  {consisting of Lipschitz contractions}. Then there exists a unique $A_{\mathcal{S}} \in \mathcal{K}^{*}(X)$ such that
$$
F_{\mathcal{S}}(A_{\mathcal{S}}, \ldots, A_{\mathcal{S}})=\bigcup_{j\in\{1,...,n\}} \phi_{j}(
\stackrel{m \; \rm{ times}}
{\overbrace{ {A_{\mathcal{S}} \times \ldots \times A_{\mathcal{S}}}}})=A_{\mathcal{S}}.
$$
Moreover, for any $A_{0}, \ldots, A_{m-1} \in \mathcal{K}^{*}(X)$, the sequence $\left(A_{k}\right)_{k\in \mathbb{N}}$, defined by $$A_{k+m}:=F_{\mathcal{S}}\left(A_{k}, \ldots, A_{k+m-1}\right),$$ $k \in \mathbb{N}$ converges to $A_{\mathcal{S}}$.
\end{theorem}

The unique compact set $A_{\mathcal{S}}$, given by Theorem~\ref{thm:GIFS-attractor-existence} is  {called} the fractal attractor  {of} the GIFS $\mathcal{S}$.

As pointed in the proof of Theorem 3.4 in \cite{Mihail2008}, we could consider in the above theorem a slightly simpler version of the fractal operator $F_{\mathcal{S}}$ as
$$
\bar{F}_{\mathcal{S}}(A)= \bigcup_{j\in\{1,...,n\}} \phi_{j}(\stackrel{m \; \rm{ times}}{\overbrace{ A \times \ldots \times A}}),
$$
for every $A \in \mathcal{B}^{*}(X)$. We notice that in the Lipschitz case, $\displaystyle\operatorname{Lip}\left(\bar{F}_{\mathcal{S}}\right) \leq \sup _{j\in\{1,...,n\}} \operatorname{Lip}\left(\phi_{j}\right)<1$.

\begin{definition}\label{def:induced-IIFS} Let $\mathcal{S}=\left(X^{m},\left(\phi_{j}\right)_{j\in\{1, ..., n\}}\right)$ be a GIFS (of order  {$m \geq 2$}). Given a set $B \in \mathcal{B}^*(X)$ we define the \emph{IIFS induced by $B$}  {with respect to} $\mathcal{S}$, as the IIFS $\mathcal{R}_{B}=(X, (\psi_{\theta})_{\theta\in \Theta}),$ where {$\Theta=B^{m-1}\times \{1,...,n\}$ and $\psi_{\theta}(x)=\phi_j(x,b_{2}, ...,b_{m} )$, for $\theta=(b_{2}, ...,b_{m},j)  \in \Theta$.}
\end{definition}
We want to study $\mathcal{S}$ from an iteration point of view by  approximating it by the induced IIFS $\mathcal{R}_{B}$. In \cite{Oliveira2017} we made an attempt of study a finite skill IFS whose attractor describes some part of the behaviour of the original GIFS  { but its attractor can be strictly contained} in the GIFS's attractor, ending with a new kind of attractor associated to a GIFS.
\begin{lemma}\label{lem:prop-induced_IIFS} Let $\mathcal{S}=(X^{m}, (\phi_j)_{j\in\{1,...,n\}})$ be a GIFS and $\mathcal{R}_{B}=(X, (\psi_{\theta})_{\theta\in \Theta}),$ be the IIFS induced by the set $B \in \mathcal{B}^*(X)$.  {Then
\begin{enumerate}
  \item The IIFS $\mathcal{R}_{B}$ is bounded;
  \item $\sup _{\theta \in \Theta} \operatorname{Lip}\left(\psi_{\theta}\right)<1$;
  \item  {If $B=A_{\mathcal{S}}$, then $F_{\mathcal{R}_B}\left(A_{\mathcal{S}}\right)=A_{\mathcal{S}}$. In particular, $A_{\mathcal{S}}$ is the attractor of $\mathcal{R}_{A_{\mathcal{S}}}$};
  \item $F_{\mathcal{R}_{B}}(A)=\bigcup_{j\in\{1,...,n\}}  {\phi_{j}(A\times \stackrel{m-1 \; \rm{ times}}{\overbrace{B\times \ldots \times B}})}=F_{\mathcal{S}}(A ,  {B, \ldots, B})$, for every\\ $A \in \mathcal{B}^{*}(X)$.
\end{enumerate}}
\end{lemma}
\begin{proof} The proof will be for $m=2$ in order to avoid unnecessarily  complex notation. Note that, from Definition~\ref{def:fractal-op-IIFS} we get
$\displaystyle F_{\mathcal{R}_{B}}(A)=\overline{\bigcup_{\theta \in \Theta} \psi_{\theta}(A)},$ for every $A \in \mathcal{B}^{*}(X)$.\\
(1) From Definition~\ref{def:bounded-family} we must show that for every bounded set $A \subset X$ the set $\bigcup_{\theta \in \Theta} \psi_{\theta}(A)$ is bounded. As $A$ is bounded we can find $a_0 \in A$ and $M_A>0$ such that $d(a,a_0) \leq  M_A$, for any $a \in A$. Analogously, we can find $b_0 \in B$ and $M_B>0$ such that $d(b,b_0) \leq  M_B$, for any $b \in B$. For each $j \in \{1,...,n\}$ we define $c_j=\psi_{(b_0,j)}(a_0)=\phi_{j}(a_0, b_0)$. For an arbitrary  {$\theta \in \Theta$ and $z \in \psi_{\theta}(A)$} we have $z=\phi_{j}(a, b)$, for some $j, a, b$ in the respective sets. It is easy to see that
$$d(z, c_j)=  d(\phi_{j}(a, b), \phi_{j}(a_0, b_0)) \leq$$ $$\leq  \operatorname{Lip}\left(\phi_{j}\right) \max( d(a, a_0), d(b, b_0)) \leq \operatorname{Lip}\left(F_{\mathcal{S}}\right) \max( M_A, M_B).$$
Let $\displaystyle M=\max_{ j\in\{1,...,n\}} d(c_1, c_j)$.  {Then, }
$$d(z, c_1) \leq d(z, c_j) + d(c_1, c_j)\leq \operatorname{Lip}\left(F_{\mathcal{S}}\right) \max( M_A, M_B)+M.$$
(2) Consider an arbitrary $\theta=(b,j) \in \Theta$ and $x,y \in X$ {. Then}
$$d(\psi_{\theta}(x), \psi_{\theta}(y))  {=} d(\phi_{j}(x, b), \phi_{j}(y, b)) \leq \operatorname{Lip}\left(\phi_{j}\right) \max( d(x, y), d(b, b))\leq  \operatorname{Lip}\left(F_{\mathcal{S}}\right) d(x, y),$$
that is, $\operatorname{Lip}\left(\psi_{\theta}\right) \leq \operatorname{Lip}\left(F_{\mathcal{S}}\right) $ for all $\theta \in \Theta$, thus
$\sup _{\theta \in \Theta} \operatorname{Lip}\left(\psi_{\theta}\right) \leq \operatorname{Lip}\left(F_{\mathcal{S}}\right)<1$, as we claimed.\\
(3) Given $B=A_{\mathcal{S}} \in\mathcal{K}^{*}(X) \subset \mathcal{B}^{*}(X)$ we know that
$$F_{\mathcal{R}_{B}}(A_{\mathcal{S}})=\overline{\bigcup_{\theta \in \Theta} \psi_{\theta}(A_{\mathcal{S}})}= \overline{\bigcup_{b \in A_{\mathcal{S}}, \; j\in\{1,...,n\}  } \phi_{j}(A_{\mathcal{S}}\times \{b\})}= \overline{\bigcup_{j\in\{1,...,n\}  } \phi_{j}( {A_{\mathcal{S}}\times A_{\mathcal{S}}})}.  $$
Since $A_{\mathcal{S}}$ is the attractor of a GIFS we know that $\bigcup_{j\in\{1,...,n\}  } \phi_{j}( {A_{\mathcal{S}}\times A_{\mathcal{S}}}) = F_{\mathcal{S}}(A_{\mathcal{S}})= A_{\mathcal{S}} \in \mathcal{K}^{*}(X)$. Substituting that in the previous computation and using the fact that $ \overline{A_{\mathcal{S}}}= A_{\mathcal{S}}$ we obtain
$$F_{\mathcal{R}_{B}}(A_{\mathcal{S}})=A_{\mathcal{S}}.$$

(4) Given $A \in \mathcal{B}^{*}(X)$ we know that
$$F_{\mathcal{R}_{B}}(A)=\overline{\bigcup_{\theta \in \Theta} \psi_{\theta}(A)}= \overline{\bigcup_{b \in B, \; j\in\{1,...,n\}  } \phi_{j}(A \times \{b\})}= \overline{\bigcup_{j\in\{1,...,n\}  } \phi_{j}( {A \times B})}=F_{\mathcal{S}}(A , B),$$ where we can get rid of the closure because we have a finite union of closed sets.
\end{proof}

From Lemma~\ref{lem:prop-induced_IIFS} (1)-(2) and Theorem~\ref{thm:exist-attract-IIFS}  we conclude that, for an arbitrary $B\in \mathcal{B}^{*}(X)$, the induced IIFS always has an attractor, denoted  $A_{\mathcal{R}_{B}}$. From this property we can define the evaluation map  {with respect to} a given GIFS $\mathcal{S}$:
\begin{definition}\label{def:evaluation-map} Let $\mathcal{S}=(X^{m}, (\phi_j)_{j\in\{1,...,n\}})$ be  {a GIFS. Define the} \emph{evaluation map} $ev_{\mathcal{S}}: \mathcal{B}^*(X) \to \mathcal{B}^*(X)$ by
\begin{equation}\label{eq:eval-map}
  ev_{\mathcal{S}}(B)= A_{\mathcal{R}_{B}} \in \mathcal{B}^*(X),
\end{equation}
for every $B \in \mathcal{B}^*(X)$, where $A_{\mathcal{R}_{B}}$ is the attractor of the induced IIFS $\mathcal{R}_{B}$.
\end{definition}

\begin{theorem}\label{thm:main-thm}
  Let $\mathcal{S}=(X^{m}, (\phi_j)_{j\in\{1,...,n\}})$ be a GIFS consisting of Lipschitz contractive maps. The evaluation map is a Lipschitz  contraction with $\operatorname{Lip}\left(ev_{\mathcal{S}}\right) \leq \operatorname{Lip}\left(F_{\mathcal{S}}\right) <1$, and for any $B_{0} \in \mathcal{B}^*(X)$, the sequence $\left(B_k\right)$ defined by $B_{k+1}=e v_S\left(B_k\right)$ for $k \geq 0$, converges to $A_{\mathcal{S}}$  in $(\mathcal{B}^*(X), h)$.
\end{theorem}
\begin{proof}
   {Let $B, B' \in\mathcal{B}^*(X)$ be such that $A_{\mathcal{R}_B} \neq A_{\mathcal{R}_{B^{\prime}}}$. From Lemma 3.4, item (4), we have:}
  $$A_{\mathcal{R}_{B}} = F_{\mathcal{R}_{B}}(A_{\mathcal{R}_{B}})= F_{\mathcal{S}}(A_{\mathcal{R}_{B}} , B, \ldots, B)$$
  and the same is true for $A_{\mathcal{R}_{B'}}$. Then,
  $$h(ev_{\mathcal{S}}(B), ev_{\mathcal{S}}(B')) = h(F_{\mathcal{S}}(A_{\mathcal{R}_{B}} , B, \ldots, B) ,\; F_{\mathcal{S}}(A_{\mathcal{R}_{B'}} , B', \ldots, B') )\leq $$
  $$\leq \operatorname{Lip}\left(F_{\mathcal{S}}\right) \max( h(A_{\mathcal{R}_{B}}, A_{\mathcal{R}_{B'}}), h(B, B'), ..., h(B, B')) = $$
  $$= \operatorname{Lip}\left(F_{\mathcal{S}}\right) \max( h(A_{\mathcal{R}_{B}}, A_{\mathcal{R}_{B'}}), h(B, B')).$$
  The last inequality follows from the fact $F_{\mathcal{S}}$ is Lipschitz with respect to the maximum distance in $\mathcal{B}^*(X)^m$.

  If $h(A_{\mathcal{R}_{B}}, A_{\mathcal{R}_{B'}}) \geq h(B, B')$ we obtain $h(A_{\mathcal{R}_{B}}, A_{\mathcal{R}_{B'}}) \leq  \operatorname{Lip}\left(F_{\mathcal{S}}\right) h(A_{\mathcal{R}_{B}}, A_{\mathcal{R}_{B'}})$, an absurd because $\operatorname{Lip}\left(F_{\mathcal{S}}\right)<1$. Thus $$\operatorname{Lip}\left(F_{\mathcal{S}}\right) \max( h(A_{\mathcal{R}_{B}}, A_{\mathcal{R}_{B'}}), h(B, B')) = \operatorname{Lip}\left(F_{\mathcal{S}}\right) h(B, B').$$ From this, we conclude that $h(ev_{\mathcal{S}}(B), ev_{\mathcal{S}}(B')) \leq \operatorname{Lip}\left(F_{\mathcal{S}}\right) h(B, B')$, in other words,  $$\operatorname{Lip}\left(ev_{\mathcal{S}}\right) \leq \operatorname{Lip}\left(F_{\mathcal{S}}\right) <1.$$   Since $ev_{\mathcal{S}}$ is a  Lipschitz contraction and from Theorem~\ref{thm:metric-spaces-of-sets} (2), $(\mathcal{B}^*(X), h)$ is complete the Banach contraction theorem claims that $B_{k+1}=ev_{\mathcal{S}}(B_{k})$ converges to the unique fixed point of $ev_{\mathcal{S}}$. On the other hand, from Lemma~\ref{lem:prop-induced_IIFS} (3) we know that $ev_{\mathcal{S}}(A_{\mathcal{S}})= A_{\mathcal{S}}$ proving that  {$(B_{k})$} converges to $A_{\mathcal{S}}$  with respect to the Hausdorff-Pompeiu metric $h$.
\end{proof}
\begin{remark}\label{rem:iteration-comparison}
    From a theoretical point of view, the iteration procedure $B_{k+1}=ev_{\mathcal{S}}(B_{k})$ from Theorem~\ref{thm:main-thm} can be seen as a kind of iteration  {for a usual GIFS}, but using the attractor of the induced IIFSs. To see that, we recall that, from Lemma~\ref{lem:prop-induced_IIFS} (4), given $B, B'\in \mathcal{B}^{*}(X)$ we know that $F_{\mathcal{R}_{B}}(B')=F_{\mathcal{S}}(B',  { B, \ldots, B}).$
    In particular, if $B'=A_{\mathcal{R}_{B}}$ then $F_{\mathcal{R}_{B}}(A_{\mathcal{R}_{B}})=A_{\mathcal{R}_{B}}$ and
    $A_{\mathcal{R}_{B}}=F_{\mathcal{S}}(A_{\mathcal{R}_{B}}, { B, \ldots, B}).$ In this way, given $B_{0} \in \mathcal{B}^*(X)$,  {we have:}\\
    $B_{1}=ev_{\mathcal{S}}(B_{0})= F_{\mathcal{S}}(A_{\mathcal{R}_{B_{0}}},  { B_{0}, \ldots, B_{0}})$;\\
    $B_{2}=ev_{\mathcal{S}}(B_{1})= F_{\mathcal{S}}(A_{\mathcal{R}_{B_{1}}},  { A_{\mathcal{R}_{B_{0}}}, \ldots, A_{\mathcal{R}_{B_{0}}}})$;\\
    $\cdots$\\
     {$ {B_{k+1}}=ev_{\mathcal{S}}(B_{k})= F_{\mathcal{S}}(A_{\mathcal{R}_{B_{k}}}, A_{\mathcal{R}_{B_{k-1}}}, \ldots, A_{\mathcal{R}_{B_{k-1}}})$.}\\
    On the other hand, by Theorem~\ref{thm:GIFS-attractor-existence} we know that for any $H_{0}, \ldots, H_{m-1} \in \mathcal{K}^{*}(X)$, the sequence $\left(H_{k}\right)_{k\in \mathbb{N}}$, defined by $H_{k+m}:=F_{\mathcal{S}}\left(H_{k}, \ldots, H_{k+m-1}\right),$ $k \in \mathbb{N}$ converges to $A_{\mathcal{S}}$.
\end{remark}

\subsection{Hutchinson (invariant) measures of GIFS}
 From now on we assume that $(X,d)$ is compact.  {The set of all Borel positive finite measures $\mu$ over the Borel sigma algebra of the metric space $(X,d)$ is denoted by $\mM(X)$. Recall that the support of a measure $\mu$ is given by $$\supp(\mu)=\{x \in X | \mu(U)>0 \text{ for any open neighborhood of } x\},$$ and that it is a closed subset of $X$.} Let $\mM_{1}(X)$  {be} the elements of $\mM(X)$ that are normalized ($\mu(X)=1$), that is, the set of all  {Borel} probability measures over $(X,d)$.

We introduce the Monge-Kantorovich metric $d_{MK}$ in $\mM_{1}(X)$ in the following way:
for every $\mu,\nu\in\mM_{1}(X)$, define
\begin{equation}
d_{MK}(\mu, \nu)=\sup \left\{\left|\int_{X} f \mathrm{d} \mu-\int_{X} f \mathrm{d} \nu\right| : f \in \operatorname{Lip}_{1}(X, \mathbb{R})\right\},
\end{equation}
where $\operatorname{Lip}_{1}(X, \mathbb{R})$ is the set of maps $f:X\to\R$ with $\on{Lip}(f)\leq 1$.  {In this case, the Monge-Kantorovich metric induces the topology of weak convergence of measures on $\mM_{1}(X)$ (see \cite{Bogachev207} for details).} From now on, we consider the complete metric  {spaces} $(\mM_{1}(X), d_{MK})$ or $(\mM_{1}(X)^{m}, d_{MK}^{m})$, where
$$d_{MK}^{m}((\mu_0,...,\mu_{m-1}), (\mu_0,...,\mu_{m-1}))=\max_{0\leq i \leq m-1} d_{MK}(\mu_{i}, \nu_{i}), \; m \geq 1.$$

 {A key improvement from the classical study of IFS with probabilities (IFSp for short) was given by Stenflo (see \cite[Remark 3]{Stenflo2002}), where random iterations are used to represent the iterations of a so called IFS with probabilities, $\mathcal{R}=(X, \psi_\theta, p)_{\theta \in \Theta}$ for an arbitrary measurable space $\Theta$. The approach here is slightly different.
\begin{definition}\label{def:IIFSp Mendivil}
       Let $\Theta$ be a compact set. An iterated function system with probabilities (IIFSp for short) $\mathcal{R}=(X, (\psi_{\theta})_{\theta\in \Theta}, p)$, is an IIFS endowed with a probability $p$ on $\Theta$, such that the map $(\theta, x) \to \psi_{\theta}(x)$ is continuous in both $\theta$ and $x$.
     \end{definition}}
We denote by $C(X,\R)$, the set of all continuous functions from $X$ to $\R$.
\begin{definition}\cite[Section 2]{mendivil1998generalization}
       Let  $\mathcal{R}=(X, (\psi_{\theta})_{\theta\in \Theta}, p)$ be an IIFSp. The \emph{transfer operator}\\ $L_{\mathcal{R}}: C(X,\R) \to C(X,\R)$ is given by
\begin{equation}\label{eq:stenflo_transf_op}
  L_{\mathcal{R}}(f)(x)=\int_{\Theta} f(\psi_{\theta}(x)) dp(\theta),
\end{equation}
for any $f \in C(X,\R)$.
     \end{definition}
  {\begin{definition}\cite[Section 2]{mendivil1998generalization}
        Let  $\mathcal{R}=(X, (\psi_{\theta})_{\theta\in \Theta}, p)$ be an IIFSp. The \emph{Markov operator} associated to $\mathcal{R}$ is the operator $M_{\mathcal{R}}: \mM_{1}(X) \to \mM_{1}(X)$ defined by:
\begin{equation}\label{eq:stenflo_markov_op}
  \int_{X} f(x) dM_{\mathcal{R}}(\nu) =\int_{X} L_{\mathcal{R}}(f)(x) d\nu,
\end{equation}
for any $f \in C(X,\R)$, $\nu \in \mM_{1}(X)$.
\end{definition}}

\begin{theorem} \cite[Theorem 1]{mendivil1998generalization} \label{thm:exist_measure_IIFS-Mendivil} Let $X$ and $\Theta$ be compact metric spaces and $\mathcal{R}=(X, (\psi_{\theta})_{\theta\in \Theta}, p)$, an IIFSp, which is contractive on average i.e., for all $x, y \in X$
$$
\int_{\Theta} d\left(\psi_{\theta}(x), \psi_{\theta}(y)\right) dp(\theta) \leq \lambda \, d(x, y)
$$
with $\lambda<1$, then  the Markov operator $M_{\mathcal{R}}: \mM_{1}(X) \to \mM_{1}(X)$ defined by
$$\int_{X} f(x) d M_{\mathcal{R}}(\mu)(x)= \int_{X} \int_{\Theta }f(\psi_{\theta}(x)) dp(\theta) d\mu(x),
$$
for any $f \in C(X, \mathbb{R})$, is contractive in the Monge-Kantorovich metric with $\on{Lip}(M_{\mathcal{R}})\leq \lambda <1$. In particular, for any initial measure  {$\nu_0 \in \mM_{1}(X)$} the sequence  $(M_{\mathcal{R}}^{k}(\nu_0))$ converges to $\mu_{\mathcal{R}}$ as $k \to \infty$.
\end{theorem}

{The unique measure $\mu_{\mathcal{R}}$, given by Theorem~\ref{thm:exist_measure_IIFS-Mendivil} is called as the \emph{Hutchinson measure} for the IIFSp $\mathcal{R}$.
\begin{theorem} \cite[Theorem 3]{mendivil1998generalization} \label{thm:exist_attractor_IIFS-Mendivil} Under the hypothesis from Theorem~\ref{thm:exist_measure_IIFS-Mendivil}, if the family $\psi_{\theta}$ is uniformly Lipschitz contractive, that is, $\on{Lip}(\psi_{\theta})\leq \lambda <1$ for every $\theta \in \Theta$, then $\supp(\mu_{\mathcal{R}})=A_{\mathcal{R}}$, where $A_{\mathcal{R}}$ is the attractor of $\mathcal{R}$, given by Theorem~\ref{thm:exist-attract-IIFS}.
\end{theorem}}

In \cite{mihail2009hutchinson} and \cite{Mihail2009}, Miculescu and Mihail studied the counterpart of the Hutchinson measure for GIFS.
\begin{definition}
By a \emph{GIFS with probabilities} (GIFSp in short) we mean a triplet
$$
\mS=\left(X^{m},(\phi_j)_{j\in\{1,...,n\}},(q_j)_{j\in\{1,...,n\}}\right),
$$
where $(X^{m},(\phi_j)_{j\in\{1,...,n\}})$ is a GIFS and $q_1,...,q_n>0$ with $\sum_{j=1}^n q_j=1$.\\
Each GIFSp generates a map $M_\mS:\mM_{1}(X)^m\to\mM_{1}(X)$, called the \emph{generalized Markov operator}, which  {associates} to any $\mu_0,...,\mu_{m-1}\in\mM_{1}(X)$, the measure\\ $M_\mS(\mu_0,...,\mu_{m-1})$ defined by,
 \begin{equation}
\int_{X}f  \;dM_\mS(\mu_0,...,\mu_{m-1})=\sum_{j=1}^n q_j\int_{X^m}f\circ \phi_j\;d(\mu_0\times...\times \mu_{m-1}),
\end{equation}
for every continuous map $f:X\to\R$.

By the  \emph{generalized Hutchinson measure} of a GIFSp $\mS$ we mean the unique measure $\mu_\mS\in\mM_{1}(X)$ which satisfies $\mu_\mS=M_\mS(\mu_\mS,...,\mu_\mS)$ and such that for every $\mu_0,...,\mu_{m-1}\in\mM_{1}(X)$, the sequence $(\mu_k)$ defined by $\mu_{m+k}=M_\mS(\mu_{k},...,\mu_{k+m-1})$, converges to $\mu_\mS$ with respect to the Monge-Kantorovich metric.
\end{definition}

 {A map $\phi: X^m \to X$  is an \emph{$(a_1,...,a_{m})$-contraction},  if
\begin{equation*}
d(\phi(x_0,...,x_{m-1}),\phi(y_0,...,y_{m-1})) \leq \sum_{i=1}^{m}a_i d(x_{i-1},y_{i-1})
\end{equation*}
for all  $(x_0,...,x_{m-1}),(y_0,...,y_{m-1})\in X^m$, where $\sum_{i=1}^{m}a_i <1$. In particular, $\phi$ is a Lipschitz contraction with $\operatorname{Lip}(\phi) \leq \sum_{i=1}^{m}a_i <1$.}

 {As proved in \cite{mihail2009hutchinson} and \cite{Mihail2009}}, if a GIFSp $\mS$ consists of $(a_1,...,a_{m})$-contractions, then $M_\mS$ is also  {an} $(a_1,...,a_{m})$-contraction.

As previously, we could redefine for proof purposes, $$\overline{M}_\mS(\mu):=M_\mS(\mu,...,\mu)$$ for each $\mu\in\mM_{1}(X)$ and consider $\overline{M}_\mS$ instead of $M_\mS$ because they have the same fixed point and Lipschitz constant. Under the above hypothesis   {$\overline{M}_\mS$ and $M_\mS$ are Banach contractions} with the Lipschitz constant $\on{Lip}(M_\mS)\leq \sum_{i=1}^{m}a_i <1$.

 {In consequence, Miculescu and Mihail} proved the following theorem (see also  \cite{DaCunha2020} for additional details):
\begin{theorem}[\cite{mihail2009hutchinson}, \cite{Mihail2009}]   \label{thm:huthc-measure-GIFSp-existence}
Assume that $\mS$ is a GIFSp on a complete metric space consisting of $(a_1,...,a_{m})$-contractions, where $\sum_{i=1}^{m}a_i<1$. Then, $\mS$ admits the Hutchinson measure $\mu_\mS$ and $\on{supp}(\mu_\mS)=A_\mS$.
\end{theorem}
It is worth to mention that Theorem~\ref{thm:huthc-measure-GIFSp-existence} was fairly improved  {in}  {\cite[ Theorem 4.3]{Strobin2020}},  proving that $\mS$ admits the Hutchinson measure $\mu_\mS$ and $\on{supp}(\mu_\mS)=A_\mS$ under the hypothesis that each map of the GIFS is a generalized Matkowski contraction, using new techniques and code spaces.

\begin{definition}\label{def:induced-IIFSp} Let $\mathcal{S}=(X^{m}, (\phi_j)_{j\in\{1,...,n\}}, (q_j)_{j\in\{1,...,n\}})$ be a GIFSp (of order $m$) consisting of $(a_1,...,a_{m})$-contractions. Given a set $B \in \mathcal{B}^*(X)$  and a Borel probability  {$\nu \in \mM_{1}(X)$, such that, $\supp(\nu) \subseteq B$}, we define the \emph{IIFSp induced by $(B, \nu)$}  with respect to $\mathcal{S}$ as  $\mathcal{R}_{B,\nu}=(X, (\psi_{\theta})_{\theta\in \Theta}, p),$ where  $\Theta=B^{m-1}\times \{1,...,n\}$ and $\psi_{\theta}(x)=\phi_j(x,b_{2}, ...,b_{m} )$, for $\theta=(b_{2}, ...,b_{m},j)  \in \Theta$ and   $p$ is  the Borel probability on $\Theta$ given by $\displaystyle dp=\sum_{j=1}^{n} q_{j}\delta_{j} \times d\nu^{(m-1)\times}(b)$, where $\nu^{(m-1)\times}:= \nu \times \cdots \times \nu$, $m-1$ times.
\end{definition}
Notice that
$$\int_{\Theta} f(\theta) dp(\theta)= \int_{B}\ldots\int_{B} \sum_{j=1}^{n} q_{j} \; f(b_{2}, ...,b_{m},j) d\nu(b_{2}) \ldots d\nu(b_{m}),$$
for any continuous function $f: \Theta \to \mathbb{R}$.

We claim that $\mathcal{R}_{B,\nu}$ is actually an IIFSp, according to Definition~\ref{def:IIFSp Mendivil}. In order to see that, we must show that the map $(\theta, x) \to \psi_{\theta}(x)$, given by $\psi_{\theta}(x)=\phi_j(x,b_{2}, ...,b_{m})$, is continuous in both $\theta$ and $x$. Indeed, the topology of $\Theta=B^{m-1}\times \{1,...,n\}$ is the product topology induced by $(X,d)$ on the closed set $B$ and the discrete topology on $\{1,...,n\}$. As  the map $(x_1,x_{2}, ...,x_{m}) \to \phi_j(x_1,x_{2}, ...,x_{m})$ is Lipschitz continuous, for each $j$, we obtain the continuity of $(\theta, x) \to \psi_{\theta}(x)$ with respect to both variables.

\begin{remark}\label{rem:place dependent GIFS}
  We notice that the problem of finding the unique Hutchinson measure for a place dependent GIFS, obtained by Miculescu~\cite{Miculescu2014} for a particular {class} of Lipschitz contractions and leaved as an open problem in {\cite[Problem 5.3]{Strobin2020}}, for generalized contractions,  seems to be a big challenge even using our approach. If we consider a place dependent GIFS $\mathcal{S}=(X^{m}, (\phi_j)_{j\in\{1,...,n\}}, (q_j)_{j\in\{1,...,n\}})$ then the associated IIFSp will  have a place dependent measure
  $$\displaystyle dp_{x}=\sum_{j=1}^{n} q_{j}(x)\delta_{j} \times d\nu^{(m-1)\times}(b),$$
  leading to the study of iterated function systems with measures IIFSm (see \cite[Section 2]{BOS23} for details) $$\mathcal{R}_{B,\nu}=(X, (\psi_{\theta})_{\theta\in \Theta}, p_{x}),$$ with transfer operator
$\displaystyle L(f)(x)=\int_{\Theta} f\left(\psi_{\theta}(x)\right) d p_{x}(\theta)$ (\cite[Definition 2.1]{BOS23}).
As far as we know, that class of process  has not been well studied yet. Although, as we proved in \cite{BOS23} it is possible to develop some classical results, such as the thermodynamical formalism, for those systems.
\end{remark}

The next proposition shows that a GIFSp consisting of $(a_0,...,a_{m-1})$-contractions always induces an IIFSp having a Hutchinson measure.
\begin{proposition}\label{prop: existence atractor induced IIFS}
  Let $\mathcal{S}=(X^{m}, (\phi_j)_{j\in\{1,...,n\}}, (q_j)_{j\in\{1,...,n\}})$ be a GIFSp satisfying the hypothesis of Theorem~\ref{thm:huthc-measure-GIFSp-existence} and $\mathcal{R}_{B, \nu}$ be the IIFSp induced by $(B, \nu)$, according to Definition~\ref{def:induced-IIFSp}. Then, the Markov operator $M_{\mathcal{R}_{B, \nu}}$ is contractive with respect to the Monge-Kantorovich metric with $\on{Lip}(M_{\mathcal{R}_{B, \nu}})\leq  {a_1} <1$  and  there exists $\mu_{\mathcal{R}_{B, \nu}}$, the Hutchinson measure for $\mathcal{R}_{B, \nu}$. {Additionally, $\supp(\mu_{\mathcal{R}_{B, \nu}})=A_{\mathcal{R}_{B, \nu}}$ where $A_{\mathcal{R}_{B, \nu}}$ is the attractor of $\mathcal{R}_{B, \nu}$.}
\end{proposition}
\begin{proof}
   In view of Theorem~\ref{thm:exist_measure_IIFS-Mendivil}, we need to show that the induced IIFSp $\mathcal{R}_{B,\nu}$ is contractive on average.
Indeed,
$$
\int_{\Theta} d\left(\psi_{\theta}(x), \psi_{\theta}(y)\right) dp(\theta)=
\int_{\Theta} d\left(\phi_j(x,b_{2}, ...,b_{m} ), \phi_j(y,b_{2}, ...,b_{m} )\right) dp(\theta) \leq $$ $$\leq  \int_{\Theta} \left(a_{1} d(x,y) + \sum_{i=2}^{m} a_i d(b_{i},b_{i})\right) dp(\theta) = \int_{\Theta} a_{1} d(x,y) dp(\theta)= a_{1} \, d(x, y),
$$
shows that we can take $\lambda=a_1 <1$. Thus Theorem~\ref{thm:exist_measure_IIFS-Mendivil} does apply because $X$ and $\Theta=B^{m-1}\times \{1,...,n\}$ are both compact metric spaces.  As each $\phi_j$ is an $(a_1,...,a_{m})$-contraction, we obtain $$d\left(\phi_j(x,b_{2}, ...,b_{m} ), \phi_j(y,b_{2}, ...,b_{m} )\right) \leq a_{1} \, d(x, y)$$ meaning that the induced IIFS is uniformly contractive. Then, from Theorem~\ref{thm:exist_attractor_IIFS-Mendivil} we get the equality $\supp(\mu_{\mathcal{R}_{B, \nu}})=A_{\mathcal{R}_{B, \nu}}$.
\end{proof}

The next lemma shows that when the  {IIFSp} is induced by the attractor $A_\mS $ and the Hutchinson measure $\mu_\mS$  {of a GIFSp} $\mathcal{S}$, its Markov operator has $\mu_\mS$ as a fixed point.
\begin{lemma}\label{lem:fixed_measure_IIFS-GIFS}
   Under the  hypothesis of Theorem~\ref{thm:huthc-measure-GIFSp-existence},
    {if $B=A_\mS, \nu=\mu_\mS$ and $R_{A_S, \mu_S}$ is the IIFSp induced by $\left(A_S, \mu_S\right)$, we have that $M_{\mathcal{R}_{A_\mS,\mu_\mS}}(\mu_\mS)=\mu_\mS$.} In particular, $\mu_{\mathcal{R}_{A_\mS,\mu_\mS}}= \mu_\mS$.
\end{lemma}
\begin{proof}
   The proof relay on the formula for  $M_{\mathcal{R}_{B,\nu}}$ replacing $B=A_\mS $  and $\nu=\mu_\mS$.  Indeed, given $f \in C(X)$ we get
  $$M_{\mathcal{R}_{A_\mS,\mu_\mS}}(\mu_\mS)(f)=\int_{X} \int_{ {A_\mS}^{m-1}} \sum_{j=1}^{n} q_{j} f\left(\phi_{j}(x, b)\right)  d\mu_\mS^{(m-1)\times}(b) d\mu_\mS(x)=$$
  $$=\int_{X} \int_{X^{m-1}} \sum_{j=1}^{n} q_{j} f\left(\phi_{j}(x, x_{2}, \ldots x_{m})\right)  d\mu_\mS(x_2)\ldots d\mu_\mS(x_m)  d\mu_\mS(x)=$$
  $$=M_{\mS}(\mu_\mS\ldots \mu_\mS )(f)=\mu_\mS(f),$$
  thus $M_{\mathcal{R}_{A_\mS,\mu_\mS}}(\mu_\mS)=\mu_\mS$. By Proposition~\ref{prop: existence atractor induced IIFS} the induced IIFSp has a Hutchinson measure $\mu_{\mathcal{R}_{A_\mS,\mu_\mS}}$, which is the unique measure satisfying $M_{\mathcal{R}_{A_\mS,\mu_\mS}}(\mu_{\mathcal{R}_{A_\mS,\mu_\mS}})=
\mu_{\mathcal{R}_{A_\mS,\mu_\mS}}$, thus $\mu_{\mathcal{R}_{A_\mS,\mu_\mS}}= \mu_\mS$.
\end{proof}

Since we are dealing with compact metric spaces the topology induced by the metric $d_{MK}$ in  $\mM_{1}(X)$ is equivalent to the one induced by the weak convergence of measures, whose main properties are given by
\begin{theorem}(\cite[Theorem 2.1]{Billingsley1971}) \label{thm:biling weak conv}
    The following conditions are equivalent for a sequence of probabilities $(\nu_{n}) \subset \mM_{1}(X)$:
\begin{enumerate}
  \item[a)] $\nu_{n}$ is weak convergent to $\nu \in \mM_{1}(X)$, that is, $ \int_{X} f d\nu_{n} \to \int_{X} f d\nu, \; \forall f \in C(X)$;
  \item[b)] $\displaystyle\limsup_{n \to \infty} \nu_{n}(F) \leq \nu(F),$   for every closed set $F \subset X$;
  \item[c)] $\displaystyle\liminf_{n \to \infty} \nu_{n}(G) \geq \nu(G),$   for every open set $G \subset X$;
  \item[d)] $\displaystyle\lim_{n \to \infty} \nu_{n}(A) = \nu(A),$  for every  $\nu$-continuity set $A \subset X$, that is, $\nu(\partial A)=0$.
\end{enumerate}
\end{theorem}
\begin{definition}\label{def: Gamma}
   Let $\Gamma$ be the subset of $\mathcal{K}^*(X) \times \mM_{1}(X)$ defined by:
$$\Gamma:=\{(B, \nu) | B \in \mathcal{K}^*(X), \; \nu \in \mM_{1}(X) \text{ and } \supp(\nu) \subseteq B\}.$$
\end{definition}
\begin{lemma}\label{lem: closeness of Gamma}
   The metric space $(\Gamma,d_{max})$ is complete.
\end{lemma}
\begin{proof}
   To see that $\Gamma$ is closed we consider a sequence $((B_{n}, \nu_{n}))_{n \geq 1} \subset \Gamma$ such that $(B_{n}, \nu_{n}) \to (B_{0}, \nu_{0})$ with respect to the distance $d_{max}$. By the definition of $d_{max}$ we obtain that $B_{n} \to B_{0}$ with respect to the  {Hausdorff distance $h$. The same is true for the second coordinate, that is, $\nu_{n} \to \nu_{0}$ with respect to the $d_{MK}$ distance (and so with respect to the weak convergence).} It remains to show that $(B_{0}, \nu_{0}) \in \Gamma$, that is, $\nu_{0}(B_{0})=1$ or equivalently, $\supp(\nu_{0}) \subseteq B_{0}$. Suppose, by contradiction, that it is not the case. Then, there exists $x \not\in B_{0}$ such that for any open neighborhood $U_{x}$ of $x$ we get $\nu_{0}(U_{x})>0$.  {Consider $\varepsilon>0$ such that $U_{x} \cap B_{0}^{\varepsilon}=\varnothing$, where $B_{0}^{\varepsilon}=\{ z \in X\; | \; d(z, B_{0})<\varepsilon\}$.} From the convergence $B_{n} \to B_{0}$ with respect to $h$ we obtain $N_{\varepsilon} \in \mathbb{N}$ such that for any $n \geq N_{\varepsilon}$ we have $B_{n} \subset B_{0}^{ {\varepsilon}}$. Since $(B_{n}, \nu_{n}) \in \Gamma$ we know that $\supp(\nu_{n}) \subseteq B_{n}$ so $\nu_{n}(U_{x})=0$ for any $n \geq N_{\varepsilon}$, thus $\displaystyle\liminf_{n \to \infty} \nu_{n}(U_{x})=0$. By Theorem~\ref{thm:biling weak conv} (c), we have $\displaystyle\liminf_{n \to \infty} \nu_{n}(U_{x}) \geq \nu_{0}(U_{x})$, thus $\nu_{0}(U_{x})=0$, a contradiction. \\
To complete the proof we observe that the completeness of $(\Gamma,d_{max})$ is a trivial consequence of the fact that $\Gamma$ is a closed subset of a complete metric space.
\end{proof}
\begin{definition}\label{def:evaluation-map-MEASURE} Let $\mathcal{S}=(X^{m}, (\phi_j)_{j\in\{1,...,n\}}, (q_j)_{j\in\{1,...,n\}})$ be a GIFSp   {consisting of $(a_0,...,a_{m-1})$-contractions}.  We define the \emph{joint evaluation map} $EV_{\mathcal{S}}: \Gamma \to \Gamma$ by
\begin{equation}\label{eq:eval-map-probability}
  EV_{\mathcal{S}}(B, \nu)= (A_{\mathcal{R}_{B, \nu}}, \mu_{\mathcal{R}_{B, \nu}}),
\end{equation}
for every  {$(B, \nu) \in \Gamma$,} where $A_{\mathcal{R}_{B, \nu}}$ is the attractor and $\mu_{\mathcal{R}_{B, \nu}}$ is the Hutchinson measure of the induced IIFS $\mathcal{R}_{B, \nu}$ {(given by Proposition~\ref{prop: existence atractor induced IIFS}), so $(A_{\mathcal{R}_{B, \nu}}, \mu_{\mathcal{R}_{B, \nu}}) \in \Gamma$}.
\end{definition}

The first coordinate of $EV_{\mathcal{S}}$ is just $ev_{\mathcal{S}}$, which we already know is Lipschitz, by Theorem~\ref{thm:main-thm}, with $\operatorname{Lip}(ev_{\mathcal{S}}) \leq \operatorname{Lip}(\mathcal{S}) <1$. Moreover, the next theorem shows that  $EV_{\mathcal{S}}$ is also Lipschitz contractive  {with respect to} the second coordinate.
\begin{theorem}\label{thm:main_theorem_measure}
   Under the  hypothesis of Theorem~\ref{thm:huthc-measure-GIFSp-existence} the joint evaluation map $EV_{\mathcal{S}}$ is Lipschitz contractive in $\Gamma$, {with}, $\operatorname{Lip}(EV_{\mathcal{S}})  \leq  \max\left( \operatorname{Lip}(\mathcal{S}), \frac{\sum_{i=2}^{m}a_i}{1- a_1} \right) <1$. In particular,  for any $ {(B_0, \nu_0 ) \in \Gamma}$, the sequence $(EV_{\mathcal{S}}^k(B_0, \nu_0))_{k \geq 0}$ converges to $(A_\mS,\mu_\mS) \in \Gamma$.
\end{theorem}
\begin{proof}
The proof will be for $m=2$ in order to avoid complex notation.  {From Lemma~\ref{lem: closeness of Gamma} the subset $\Gamma$ is a complete metric space with respect to the metric $d_{max}$, given by $d_{max}((B, \nu), (B', \nu'))=\max(h(B,B'), d_{MK}(\nu, \nu'))$ inherited from the complete metric space  $(\mathcal{K}^*(X)\times\mM_{1}(X), d_{max})$. Our aim is to use the Banach fixed point for Lipschitz contractions in $\Gamma$.}

Denoting $\mathcal{R}:=\mathcal{R}_{B, \nu}$ and $\mathcal{R}':=\mathcal{R}_{B', \nu'}$ we need to estimate only  $d_{MK}(\mu_{\mathcal{R}}, \mu_{\mathcal{R}'})$. We recall that, for each Lipschitz function $f$, with $\operatorname{Lip}(f) \leq 1$, we have  that
$$\mu_{\mathcal{R}}(f)=\int_{X} \int_{X} \sum_{j=1}^{n} q_{j} f\left(\phi_{j}(x, b)\right)  d\nu(b) d\mu_{\mathcal{R}}(x)$$
and
$$\mu_{\mathcal{R}'}(f)=\int_{X} \int_{X} \sum_{j=1}^{n} q_{j} f\left(\phi_{j}(x, b)\right)  d\nu'(b) d\mu_{\mathcal{R}'}(x),$$
where the integration over $B$ (resp. $B'$) is replaced by $X$ because the support of $\nu$ is $B$ (resp. $\nu'$ is $B'$).

Let $f: X \rightarrow \mathbb{R}$ be such that $\operatorname{Lip}(f) \leq 1$, and define functions $g, g':X \to \mathbb{R}$ by 
$$g(x)=\int_{X} \sum_{j=1}^{n} q_{j} f\left(\phi_{j}(x, b)\right)  d\nu(b)$$ and $g'(x)=\int_{X} \sum_{j=1}^{n} q_{j} f\left(\phi_{j}(x, b)\right)  d\nu'(b)$. We also define, for each $x \in X$, the function  $r_{x}:X \to \mathbb{R}$ by $ r_{x}(b)=\sum_{j=1}^{n} q_{j} f\left(\phi_{j}(x, b)\right)$. We claim that $\operatorname{Lip}( {r_{x}}) \leq a_2$, uniformly  {with respect to} $x$, and
$\operatorname{Lip}(g') \leq a_1$ (resp. $\operatorname{Lip}(g) \leq a_1$). Indeed,
$$| {r_{x}}(b)- {r_{x}}(b')| \leq \sum_{j=1}^{n} q_{j} |f\left(\phi_{j}(x, b)\right) - f\left(\phi_{j}(x, b')\right)|\leq $$ $$\leq \sum_{j=1}^{n} q_{j} \operatorname{Lip}(f) \left(a_1\; d(x,x) + a_2 \;d(b,b')\right) \leq  a_2\; d(b,b'),$$
meaning that $\operatorname{Lip}( {r_{x}}) \leq a_2$ for all $x \in X$.

Analogously,
$$|g'(x) - g'(y)|\leq \int_{X} \sum_{j=1}^{n} q_{j} |f\left(\phi_{j}(x, b)\right)- f\left(\phi_{j}(y, b)\right)| d\nu'(b) \leq $$ $$\leq \int_{X} \sum_{j=1}^{n} q_{j} \operatorname{Lip}(f) \left(a_1\; d(x,y) + a_2\; d(b,b)\right) d\nu'(b) \leq a_1\; d(x,y), $$
meaning that $\operatorname{Lip}(g') \leq a_1$.
For each $x \in X$ we obtain
$$|g(x)-g'(x)|= \left|\int_{X}   {r_{x}}(b)   d\nu(b) - \int_{X}   {r_{x}}(b)   d\nu'(b)\right| =$$ $$= a_2  \left|\int_{X}  \frac{ {r_{x}}(b)}{a_2}   d\nu(b) - \int_{X}  \frac{ {r_{x}}(b)}{a_2}   d\nu'(b)\right| \leq a_2 \; d_{MK}(\nu,\nu'),$$
because $\operatorname{Lip}(\frac{ {r_{x}}}{a_2}) \leq 1$.

Evaluating $\left| \mu_{\mathcal{R}}(f) - \mu_{\mathcal{R}'}(f)  \right|$ we obtain
$$\left| \mu_{\mathcal{R}}(f) - \mu_{\mathcal{R}'}(f)  \right| = \left| \int_{X} g(x)  d\mu_{\mathcal{R}}(x)  -   \int_{X} g'(x)  d\mu_{\mathcal{R}'}(x) \right| \leq $$
$$\leq \left| \int_{X} g(x)  d\mu_{\mathcal{R}}(x)  -   \int_{X} g'(x)  d\mu_{\mathcal{R}}(x) \right| + \left| \int_{X} g'(x)  d\mu_{\mathcal{R}}(x)  -   \int_{X} g'(x)  d\mu_{\mathcal{R}'}(x) \right|\leq$$
$$\leq \int_{X}\left| g(x)    -     g'(x)\right|  d\mu_{\mathcal{R}}(x) + a_1 \left| \int_{X} \frac{g'(x)}{a_1}  d\mu_{\mathcal{R}}(x) -  \int_{X} \frac{g'(x)}{a_1}  d\mu_{\mathcal{R}'}(x) \right| \leq $$
$$\leq a_2\; d_{MK}(\nu,\nu')+ a_1\; d_{MK}(\mu_{\mathcal{R}},\mu_{\mathcal{R}}'),$$
because $\operatorname{Lip}(\frac{g'}{a_1}) \leq 1$. Thus
$$d_{MK}(\mu_{\mathcal{R}}, \mu_{\mathcal{R}'})=\max_{ \operatorname{Lip}(f) \leq 1} \left| \mu_{\mathcal{R}}(f) - \mu_{\mathcal{R}'}(f)  \right| \leq a_2\; d_{MK}(\nu,\nu')+ a_1\; d_{MK}(\mu_{\mathcal{R}},\mu_{\mathcal{R}}'),$$
meaning that
$\displaystyle d_{MK}(\mu_{\mathcal{R}}, \mu_{\mathcal{R}'}) \leq \frac{a_2}{1- a_1} d_{MK}(\nu,\nu').$

We now return to the map $EV_{\mathcal{S}}$:
$$d_{max}(EV_{\mathcal{S}}(B, \nu), EV_{\mathcal{S}}(B', \nu'))= d_{max}((A_{\mathcal{R}}, \mu_{\mathcal{R}}), (A_{\mathcal{R}}, \mu_{\mathcal{R}}))=$$
$$=\max(h(A_{\mathcal{R}},A_{\mathcal{R}'}), d_{MK}(\mu_{\mathcal{R}}, \mu_{\mathcal{R}'})) \leq $$ $$ \leq \max\left( \operatorname{Lip}(\mathcal{S}) h(B,B'), \frac{a_2}{1- a_1} d_{MK}(\nu,\nu')\right)= $$
$$= \max\left( \operatorname{Lip}(\mathcal{S}), \frac{a_2}{1- a_1} \right) \max(  h(B,B'), d_{MK}(\nu,\nu')) =$$ $$= \max\left( \operatorname{Lip}(\mathcal{S}), \frac{a_2}{1- a_1} \right) d_{max}((B, \nu), (B', \nu')),$$
meaning that $\operatorname{Lip}(EV_{\mathcal{S}})  \leq  \max\left( \operatorname{Lip}(\mathcal{S}), \frac{a_2}{1- a_1} \right) <1$  {(recall that $\sum_{i=1}^{2} a_i <1$).} In order to conclude our proof, we just notice that {since $EV_{\mathcal{S}}$ is contractive for any $(B_0, \nu_0 ) \in \Gamma$, the sequence $(EV_{\mathcal{S}}^k(B_0, \nu_0))_{k \geq 0}$ converges to $(\overline{A},\overline{\mu}) \in \Gamma$ as $k \to \infty$, where $(\overline{A},\overline{\mu}) \in \Gamma$ is the unique fixed point of $EV_{\mathcal{S}}$. From Lemma~\ref{lem:fixed_measure_IIFS-GIFS} we obtain $EV_{\mathcal{S}}(A_\mS,\mu_\mS)= (A_\mS,\mu_\mS)$, thus $(\overline{A},\overline{\mu})= (A_\mS,\mu_\mS)$, concluding our proof.}
\end{proof}

\section{Concluding remarks and future work}\label{sec:concluding-remarks}
\subsection{Revisiting the approximation procedure}
Theorem~\ref{thm:main-thm} provides an approximation procedure to obtain the attractor of a GIFS via attractors of IIFSs. Indeed, given $B_0=\{x_0\}$  we obtain $B_1=ev_{\mathcal{S}}(B_{0})=A_{\mathcal{R}_{B_0}}$ via iteration of an initial set $H \subset \mathcal{K}^*(X)$  as in Theorem~\ref{thm:exist-attract-IIFS}.   {Then } we pick $B_2=ev_{\mathcal{S}}(B_{1})=A_{\mathcal{R}_{B_1}}$ again via iteration of an initial set $H \subset \mathcal{K}^*(X)$,   as in Theorem~\ref{thm:exist-attract-IIFS} and so on, obtaining that $h\left(B_k, A_\mathcal{S}\right) \rightarrow 0$. We have proved that $B_{k} \simeq A_{\mathcal{S}}$, meaning that $h(B_{k}, A_{\mathcal{S}}) \to 0$ as $k \to \infty$. For a GIFSp, from Theorem~\ref{thm:main_theorem_measure}, if we start with $(B_0=\{x_0\}, \nu_0=\delta_{x_0}) \in \mathcal{K}^*(X)\times\mM_{1}(X)$ we produce $(B_1, \nu_1)=EV_{\mathcal{S}}(B_0, \nu_0)=(A_{\mathcal{R}_{0}},\mu_{\mathcal{R}_{0}})$ where $\mathcal{R}_{0}:=\mathcal{R}_{B_0,\nu_0}$ is the IIFSp induced by the pair {$(B_0, \nu_0)$.} Naturally, $\supp(\nu_1) \subseteq B_1$  {from Theorem~\ref{thm:exist_measure_IIFS-Mendivil}.  {Then successively } we choose $\mathcal{R}_{1}:=\mathcal{R}_{B_1,\nu_1}$ the IIFSp induced by the pair $(B_1, \nu_1)$} and produce $(B_2, \nu_2)=EV_{\mathcal{S}}(B_1, \nu_1)=(A_{\mathcal{R}_{1}},\mu_{\mathcal{R}_{1}})$ and so on. The sequence $(B_k,\nu_k):=EV_{\mathcal{S}}^k(B_0, \nu_0) \to (A_\mS,\mu_\mS)$ as $k \to \infty$.

\subsection{Approximation algorithms}
One could ask if there exist effective algorithms to approximate the attractor $A_{\mathcal{R}_{B}}$ of the IIFS $\mathcal{R}_{B}$ and if we can use it to approximate $A_{\mathcal{S}}$.  For countable IFSs \cite[Theorem 5]{Sec01} shows that for a given compact subset $A$ of a metric space, is possible to construct
a countable iterated function system having $A$ as its attractor. The construction involves sequence of iterated function systems whose attractors
approximate $A$. For non countable  {IIFSs} \cite{mantica2010dynamical} proved some results on the approximation of measures generated by uncountably many one-dimensional affine maps.

The theoretical procedure described in Remark~\ref{rem:iteration-comparison} is quiet difficult in the practical  {use}, since it requires the computation of many attractors of infinite  {IFSs} in order to approximate $A_{\mathcal{S}}$. A possible scheme to perform this, in a reasonable way, is the following: we could start with $B_{0} \in \mathcal{K}^*(X)$ and define $\mathcal{R}_{0}:=\mathcal{R}_{B_{0}}$. We know that $F_{\mathcal{R}_{0}}(A_{\mathcal{R}_{0}})=A_{\mathcal{R}_{0}}$ and for any $H_{0} \in \mathcal{K}^*(X)$ the sequence $H_{k+1}=F_{\mathcal{R}_{0}}(H_{k})= F_{\mathcal{S}}(H_{k}, B_{0})$ converges to $A_{\mathcal{R}_{0}}$. Instead, we introduce an approximated iteration process. Since $B_0$ is compact we can find, for any $\beta_1>0$ a finite set $B_{0}^{\beta_{1}}$ such that $h(B_{0},B_{0}^{\beta_{1}})< \beta_{1}$, denoted $\beta_{1}$- {approximation of} $B_{0}$, and define the finite IFS $\mathcal{R}_{0}^{\beta_{1}}:=\mathcal{R}_{B_{0}^{\beta_{1}}}$ whose attractor $A_{\mathcal{R}_{0}^{\beta_{1}}}$ can be easily approximated with arbitrary precision $\sigma_{1}>0$. Let $B_1 \in \mathcal{K}^*(X)$  {be such that $h(B_1 , A_{\mathcal{R}_{0}^{\beta_{1}}})< \sigma_{1}$.} Since $B_1$ is compact we can find, for any $\beta_2>0$ a finite set $B_{1}^{\beta_{2}}$ such that $h(B_{1},B_{1}^{\beta_{2}})< \beta_{2}$, denoted $\beta_{2}$- {approximation of} $B_{1}$, and define the finite IFS $\mathcal{R}_{1}^{\beta_{2}}:=\mathcal{R}_{B_{1}^{\beta_{2}}}$ whose attractor $A_{\mathcal{R}_{1}^{\beta_{2}}}$ can be easily approximated with arbitrary precision $\sigma_{2}>0$.    { In this way we obtain next sets $B_2, B_3, \ldots$ which approximate $A_{\mathcal{S}}$.} This can be synthesized as an algorithm:
\begin{figure}[H]
{\tt
\begin{tabbing}
aaa\=aaa\=aaa\=aaa\=aaa\=aaa\=aaa\=aaa\= \kill
    \>  Algorithm 1 \\
    \> Consider a GIFS $\mathcal{S}=(X^{m}, (\phi_j)_{j\in\{1,...,n\}})$.\\
    \> Input: Choose sequences $\beta_{k}, \sigma_{k} \to 0$.\\
    \> Input: Choose $B_{0} \in \mathcal{K}^*(X)$, finite.\\
    \> Input: Consider $\mathcal{R}_{0}:=\mathcal{R}_{B_{0}}$.\\
    \> Input: Choose $N \in \mathbb{N}$ the number of iterations.\\
     \> Output: A compact set $B_N$ which approximate $A_{\mathcal{S}}$.\\
     \>  {\bf for } $k$ {\bf from } $1$ to $N$ {\bf do } \\
     \> \> Choose $B_{k-1}^{\beta_{k}}$ a $\beta_{k}$- {approximation of} $B_{k-1}$. \\
     \> \> Define the finite IFS $\mathcal{R}_{k-1}^{\beta_{k}}:=\mathcal{R}_{B_{k-1}^{\beta_{k}}}$. \\
     \> \>  Choose $B_{k} \in \mathcal{K}^*(X)$ such that $h(B_{k}, A_{\mathcal{R}_{k-1}^{\beta_{k}}})< \sigma_{k}$.\\
     \>  {\bf end loop}
\end{tabbing}}
\caption{Algorithm to approximate $A_{\mathcal{S}}$ by attractors of IIFSs.} \label{fig:algo}
\end{figure}
 {To see that this algorithm works}, we recall a result on the stability of the procedure of successive approximations for Banach contractive maps from \cite{Ostr1967}, or more recently, from \cite{Jachymski1997} for generalized contractions. The stability of iterations is given by the Ostrowski's theorem:
\begin{theorem}[Ostrowski]\label{yhm:ostrowski-classic}
Let $(X, d)$ be a complete metric space and $T: X \rightarrow X$ be a contraction with $\operatorname{Lip}(T) \leq \alpha<1$. Let $\left(\varepsilon_n\right)$ be a sequence of positive real numbers and $\left(y_n\right) \subset X$ be such that
$$
y_0=x_0 \text{ and } d\left(y_k, T\left(y_{k-1}\right)\right) \leq \varepsilon_k \text{ for } k \in \mathbb{N}.
$$
Then
\begin{equation}\label{eq:Ostrowski-estimative}
        d\left(y_{k}, p \right) \leq \frac{\alpha^{k}}{1-\alpha} \; d\left(x_{0}, T x_{0}\right)+\sum_{i=1}^{k} \alpha^{k-i} \; \varepsilon_{k}, \; \text{ for } k \in \mathbb{N},
    \end{equation}
where $p$ is the fixed point of $T$.
    In particular, if $\varepsilon_{k} \rightarrow 0$, then $y_{k} \rightarrow p$.
\end{theorem}

\begin{theorem}\label{thm:algorithm works}
     {The Algorithm~\ref{fig:algo}} approximates the attractor $A_{\mathcal{S}}$ of the GIFS $$\mathcal{S}=(X^{m}, (\phi_j)_{j\in\{1,...,n\}}),$$ that is, the obtained sequence  {$(B_{k})$} converges to $A_{\mathcal{S}}$ as $k \to \infty$.
\end{theorem}
\begin{proof}
  The idea is to use Theorem~\ref{yhm:ostrowski-classic}. In order to do that we choose, $T=ev_{\mathcal{S}}$, $p=A_{\mathcal{S}}$, $x_0=B_0$, $y_{k}=B_{k}$,  $\alpha=\operatorname{Lip}(ev_{\mathcal{S}}) \leq \operatorname{Lip}(F_{\mathcal{S}})<1$, $\varepsilon_{k}=\alpha \beta_{k}+\sigma_{k} \to 0$ {. We only need} to show that $h(ev_{\mathcal{S}}(B_{k}), B_{k}) \leq \varepsilon_{k}$. Indeed,
$$h(ev_{\mathcal{S}}(B_{k}), B_{k}) \leq h(ev_{\mathcal{S}}(B_{k}), ev_{\mathcal{S}}(B_{k-1}^{\beta_{k}}))+h(ev_{\mathcal{S}}(B_{k-1}^{\beta_{k}}), B_{k})\leq$$
$$\leq \alpha  h(B_{k}, B_{k-1}^{\beta_{k}}) + h(A_{\mathcal{R}_{k-1}^{\beta_{k}}}, B_{k})\leq \alpha \beta_{k}+\sigma_{k}=\varepsilon_{k}.$$
\end{proof}

We notice that the step  {(3) of the loop} in the Algorithm 1 consists in to approximate the attractor of a finite IFS, and there are many efficient ways to do that. The only computational restriction is the number of maps in $\mathcal{R}_{k-1}^{\beta_{k}}$, which is at most $n \times \sharp \{(B_{k-1}^{\beta_{k}})^{m-1}\}$. This number increases when $\beta_{k} \to 0$ but is always much smaller than the one necessary in the iteration of the original $F_{\mathcal{S}}$. For practical purposes one can choose $\beta_{k}=\sigma_{k}=\frac{1}{k}$ and then $\varepsilon_{k}=\frac{1+\alpha}{k} \to 0$.  We presented a pseudocode here, but the implementation of an actual algorithm would be the subject of a future work employing the discrete algorithm from \cite{daCunhaMultiAlgor2021} for the step (3) of the loop and making a comparison with the classical iteration for GIFS.

\subsection{Approximate Chaos Game Theorem and Ergodic theorem for GIFS}
Another natural question is if there exists some natural chaos game theorem or  ergodic theorem for the induced IIFS which could approximate, in a reasonable way, the attractor  and integrals  with respect to the Hutchinson measure of a given GIFS. We notice that a GIFS $\mathcal{S}=(X^{m}, (\phi_j)_{j\in\{1,...,n\}})$ is not a dynamical object, meaning that, from an initial  {$m$-tuple} $(x_{0}, ..., x_{m-1}) \in X^{m}$ and a $j_0 \in\{1,...,n\}$ we obtain a single value $x_{m}=\phi_{j_0}(x_{0}, ..., x_{m-1})$, but there is no obvious recipe to continue the iteration process. In \cite{Oliveira2017}  we proposed a process where $x_{m+1}=\phi_{j_1}(x_{1}, ..., x_{m})$, for $j_1 \in\{1,...,n\}$, and so on. But this process is not capable to describe the actual attractor $A_{\mathcal{S}}$, only a smaller set  (see \cite[Example 11]{Oliveira2017}).  {Unlike GIFSs, IIFSs are dynamically defined, meaning that they can} be iterated from an initial point forming an orbit.  {Given a set $B \in B^*(X)$, let $\mathcal{R}_{B}=(X, (\psi_{\theta})_{\theta\in \Theta}),$ where  {$\Theta=B^{m-1}\times \{1,...,n\}$ and $\psi_{\theta}(x)=\phi_j(x,b_{2}, ...,b_{m} )$, for $\theta=(b_{2}, ...,b_{m},j)  \in \Theta$}, be the induced IIFS. Given $x_0 \in X$ and $\theta_0 =(b_{2}^{0}, ...,b_{m}^{0}, j_0)$, define:}
$$x_{1}=\psi_{\theta_{0}}(x_{0})= \phi_{j_{0}}(x_{0},b_{2}^{0}, ...,b_{m}^{0}).$$ {Then, choose  $\theta_{1}=(b_{2}^{1}, ...,b_{m}^{1}, j_{1})$ and define:
$$x_{2}=\psi_{\theta_{1}}(x_{1})= \phi_{j_{1}}(x_{1},b_{2}^{1}, ...,b_{m}^{1}),$$
and so on. In general, define}
$$x_{k+1}=\psi_{\theta_{k}}(x_{k})= \phi_{j_{k}}(x_{k},b_{2}^{k}, ...,b_{m}^{k}), \; k \geq 0.$$
This sequence has a lot more freedom to spread than the one used in \cite{Oliveira2017}, because at each iteration $\theta_{k}=(b_{2}^{1}, ...,b_{m}^{1}, j_{k})$ is chosen accordingly a probability $p$.

 {We recall that an IFS has the chaos game property, if under some suitable hypothesis (see \cite[Theorem 1]{Barnsley2011}, also \cite{barnsley2014chaos} for a topological point of view), given a random orbit $\left\{x_k\right\}_{k=0}^{\infty}$ of $x_0$ under an IFS $\mathcal{F}$, then, with probability one,
$$
A_{\mathcal{F}}=\lim _{K \rightarrow \infty}\left\{x_k\right\}_{k=K}^{\infty}
$$
where the limit is with respect to the Hausdorff metric and $A_{\mathcal{F}}$ is the attractor of $\mathcal{F}$.} The first question is, if the process $\{(x_{i}^{k})_{i \geq 0}\}$  has the chaos game property for $\mathcal{R}_{B_{k}}$  then $$\displaystyle\lim_{k \to \infty} \overline{\{(x_{i}^{k})_{i \geq k}\}}=A_{\mathcal{S}}?$$

The second question is, since $\displaystyle \lim_{k \to \infty} A_{\mathcal{R}_{B_{k}}}=A_{\mathcal{S}}$ and $\displaystyle \lim_{k \to \infty} \mu_{\mathcal{R}_{B_{k}}}=\mu_{\mathcal{S}}$, if the process the process $(x_{i}^{k})_{i \geq 0}$  is ergodic for $\mathcal{R}_{B_{k}}$, that is, for almost (in the sense of Theorem~\ref{thm:ergodic}) all address sequences $\gamma=(\theta_{0},\theta_{1},...) \in \Theta^{\mathbb{N}}$ we have
$\displaystyle
\lim _{N \to \infty} \frac{1}{N} \left( f\left( x_{0}^{k}\right) + \cdots + f\left( x_{ {N}-1}^{k}\right) \right)=\int_{X} f(y) d\mu_{\mathcal{R}_{B_{k}}}(y),
$
for any $f \in C(X, \mathbb{R})$, then
$$
\lim _{k\to \infty}\lim _{N \to \infty} \frac{1}{N} \left( f\left( x_{0}^{k}\right) + \cdots + f\left( x_{ {N}-1}^{k}\right) \right)=\int_{X} f(y) d\mu_{\mathcal{S}}(y)?
$$

A Chaos Game result for IIFS was proved by Le{\'{s}}niak in \cite{Lesniak-2015}, but only for countable IIFS and an ergodic theorem for IIFS is given in \cite{hyong2005ergodic}
\begin{theorem}\cite[Ergodicity of IIFS]{hyong2005ergodic}\label{thm:ergodic}
 {Let $\mathcal{R}=(X, (\psi_{\theta})_{\theta\in \Theta},p)$ be an IIFSp, where $(X,d)$ is a compact metric space, $\Theta$ is compact,  $p$ is a Borel probability on $\Theta$ and $P$ is the product measure induced by $p$ in $\Theta^{\mathbb{N}}$.} If $\mathcal{R}$ is bounded, uniformly contractive ($\int_{\Theta }\operatorname{Lip}(\psi_{\theta}) dp(\theta)<1$), and $\mu$ is the  {Hutchinson} measure of the IIFSp then, for any $f \in C(X, \mathbb{R})$ and $\forall x \in X$, for $P$ almost all address sequences $\gamma=(\theta_{0},\theta_{1},...) \in \Theta^{\mathbb{N}}$ we have
$$
\lim _{n \rightarrow \infty} \frac{1}{n} \sum_{ {1} \leqslant m \leqslant n} f\left( \left(\psi_{\theta_{m-1}} \circ  \cdots \cdot \circ  \psi_{\theta_{0}}\right)(x)\right)=\int_{X} f(y) d\mu(y).
$$
\end{theorem}
Another computational ergodic theorem was proved for an IIFS in \cite{Systems2008}, but only for $\Theta=\mathbb{N}$.  To answer these questions using the above results, or others like those, would be the subject of a future work.

\subsection{Further generalizations and the respective induced IIFS}
The family of sets which are attractors  {of GIFSs is wider than the one formed by attractors of finite IFSs. However,} we proved that all this fractal attractors are also attractors of IIFS and they are also well approximated by them.Observe that if $B$ is closed and bounded, then it is the attractor of the IIFS $\mathcal{R}=(X, (\psi_{\theta})_{\theta\in \Theta})$, where $\Theta=B$ and   $\psi_{\theta}(x)=\theta$ for all $x \in X$. Hence the matter is to define IIFS with certain properties {, which generates a given set $B$.} An effort  towards finding more general fractals which are not attractors of any known IIFSs (and so of GIFSs) is to define and study infinite GIFSs. This was done in \cite{DISF2015} for a topological version of possibly infinite GIFS as a Matkowski function system, that is, a compact-to-compact family of mappings which are  uniformly generalized Matkowski contractions.

\begin{theorem}\cite[Theorem 3.5]{DISF2015}
Assume that $(X, d)$ is a complete metric space and $$\mathcal{S}=\left(X^{m},\left(\phi_{j}\right)_{j\in\mathcal{F}}\right)$$ is a GIFS that satisfies the following conditions:\\
(i) $\sup_{j \in \mathcal{F}} \operatorname{Lip}\left(\phi_j\right)<1$, and\\
(ii) The fractal operator, associated to $\mathcal{S}$ is compact-to-compact, that is, preserves compact sets: for any $A_{1}, \ldots, A_{m} \in \mathcal{K}^*(X)$
$$F_{\mS}(A_{1}, \ldots, A_{m}) =\overline{\bigcup_{j\in\mathcal{F}}\phi_{j}(A_{1}\times \ldots\times A_{m})}\in \mathcal{K}^*(X).$$ Then, $\mathcal{S}$ generates a unique {attractor}.
\end{theorem}
For such systems we can also employ our approach producing IIFSs induced by each compact set. More precisely, for a possibly infinite GIFS on $X$,  $\mathcal{S}=\left(X^{m},\left(\phi_{j}\right)_{j\in\mathcal{F}}\right)$, of order $m \geq 2$ and a set $B \in \mathcal{K}^*(X)$, the induced IIFS will be $\mathcal{R}_{B}=(X, (\psi_{\theta})_{\theta\in \Theta})$, where $\Theta=B^{m-1}\times \mathcal{F} $ and $\psi_{\theta}(x)=\phi_j(x,b_{2}, ...,b_{m} )$, for $\theta=(b_{2}, ...,b_{m},j)  \in \Theta$.

Finally, following the program of expansion and generalization of the families of sets which are fractals generated by  {IFSs one could consider (finite) GIFSs} with infinite order (denoted GIFS$_{\infty}$). Such construction  { appeared} in \cite{Maslanka2020} inspired by Seceleans's approach \cite{Secelean2014}, showing that a typical compact set in  {a Polish metric space} is a generalized fractal.  This result shows that by considering GIFS$_{\infty}$ we can describe significantly more sets than using  {classical} IFS theory. Actually, \cite{Sec01} presented, for each compact subset $K$ of a metric space, the construction of a countable iterated function system (CIFS) having $K$ as a fractal attractor. Secelean in \cite{Secelean2014} considered mappings defined on the space $(\ell_{\infty}(X), d_{\infty})$ of all bounded sequences of elements from $X$ with values in $X$, endowed with the supremum metric $d_{\infty}$, where $(X, d)$ is a metric space.

\begin{definition}
A \emph{generalized iterated function system of infinite order} (GIFS$_{\infty}$ in short)  $$\mathcal{S}=\left(\ell_{\infty}(X), \left(\phi_{j}\right)_{j\in\{1, ..., n\}}\right),$$ consists of a finite family of  {continuous} functions $\phi_{j}: \ell_{\infty}(X) \rightarrow X$.
\end{definition}
We say that $\mathcal{S}$ satisfy the compact closure property, if for every $\left(K_{k}\right) \in \ell_{\infty}(\mathcal{K}^*(X))$, the closure of the image of the product $\overline{\phi\left(\prod_{k=1}^{\infty} K_{k}\right)} \in \mathcal{K}^*(X).$

\begin{theorem}( {\cite[Theorem 3.7]{Secelean2014} or \cite[Theorem 2.5]{Maslanka2020}})
Let $(X,d)$ be a complete metric space and $\mathcal{S}$ be a GIFS$_{\infty}$ {consisting of Banach contracting maps ($\operatorname{Lip}\left(\phi_j\right)<1, \; 1\leq j\leq n$)}. If $\mathcal{S}$ satisfy the compact closure  property, then there is a unique $A_{\mathcal{S}} \in \mathcal{K}^*(X)$ such that
$$
A_{\mathcal{S}}=\bigcup_{j=1}^{n} \overline{\phi_{j}\left(\prod_{k=1}^{\infty} A_{\mathcal{S}}\right)}.$$
\end{theorem}
The set $A_{\mathcal{S}}$ is called the fractal or the attractor of the GIFS$_{\infty}$  $\mathcal{S}$.

To make a complete and up to date reference on recent developments regarding GIFS$_{\infty}$ we notice that Ma{\'{s}}lanka and Strobin \cite{MASLANKA20181795}, made a significative advance on this subject, studying some further aspects of Secelean’s setting. More precisely, the attractor of a GIFS$_{\infty}$ is approximated by attractors of GIFS of order $m$ when it increases, {\cite[Theorem 4.11]{MASLANKA20181795}}, assuming only that the GIFS maps are generalized Banach contractions. For the last, they present, in Section 7 of \cite{MASLANKA20181795}, a Cantor set on the plane which is an attractor of some GIFS$_{\infty}$, but cannot be generated by any GIFS, reinforcing the wider range of this theory regarding new fractals creation capability.

One more time, for a given GIFS$_{\infty}$, satisfying reasonable assumptions, we could investigate the induced IIFS relating its attractors with $A_{\mathcal{S}}$. More precisely, for a GIFS$_{\infty}$ on $X$, given by  $\mathcal{S}=\left(\ell_{\infty}(X), \left(\phi_{j}\right)_{j\in\{1, ..., n\}}\right)$, the induced IIFS  {would} be $\mathcal{R}_{B}=(X, (\psi_{\theta})_{\theta\in \Theta})$, where  {$\Theta=\ell_{\infty}(B) \times \{1,...,n\}$ and $\psi_{\theta}:X \to X$ is given by $\psi_{\theta}(x)=\phi_j(x,b_{2},b_{3}, ...)$, for $\theta=((b_{2}, b_{3},...),j)  \in \Theta$.}
\subsection*{Acknowledgments}
I would like to thanks professor Filip Strobin for its valuable suggestions and conversations who has truly improved this manuscript.

\end{document}